\documentclass{amsart}
\usepackage[english]{babel}
\usepackage[T1]{fontenc}
\usepackage[color,matrix,arrow]{xy}
\usepackage{tikz}
\usetikzlibrary{decorations.markings,arrows,automata,arrows,backgrounds,snakes}
\usepackage{amsmath,amssymb,amsthm,amsfonts,amscd,mathrsfs,tikz-cd}
\usepackage{listings}
\usepackage{color}
\usepackage{algpseudocode}
\usepackage{algorithm}
\usepackage{comment}

\usepackage{amsmath}
\usepackage{amssymb}
\usepackage{amsfonts}
\usepackage{graphicx}
\usepackage[colorinlistoftodos]{todonotes}
\usepackage[colorlinks=true, allcolors=blue]{hyperref}
\usepackage{float}
\usepackage{mathtools}
\usepackage{amssymb,amsmath}
\usepackage{tikz}
\usetikzlibrary{shapes.geometric}
\usepackage{amsthm}
\theoremstyle{definition}

\newcommand{\Aut}    {\mathrm{Aut}}

\newcommand{\Ker}   {\mathrm{Ker}}
\newcommand{\id}   {\mathrm{id}}
\newcommand{\M}    {\mathcal{M}}
\newcommand{\HH}    {\mathcal{H}}

\newtheorem{thm}{Theorem}

\newcommand{\pd}    {\partial\Delta}

\newtheorem{lem}[thm]   {Lemma}
\newtheorem{cor}[thm]   {Corollary}
\newtheorem{rem}[thm]   {Remark}
\newtheorem{defn}[thm]  {Definition}
\newtheorem{ex}[thm]    {Example}
\newtheorem{prop}[thm]  {Proposition}

\newcounter{foo}  \Alph{foo}

\newenvironment{example}
{\medskip\par\noindent{\sc Example}\ }
{\par}

\begin{document}
\tikzset{->-/.style={decoration={
  markings,
  mark=at position .5 with {{->}}},postaction={decorate}}}

\title{On the homotopy and strong homotopy type of complexes of discrete Morse functions}
\author{Connor Donovan}
\author{Maxwell Lin}
\author{Nicholas A. Scoville}

\address{Department of Mathematics and Computer Science, Ursinus College, Collegeville PA 19426}
\email{codonovan@ursinus.edu}

\address{Department of Mathematics, Berkeley, CA 94720}

\email{mxlin@berkeley.edu}

\address{Department of Mathematics and Computer Science, Ursinus College, Collegeville PA 19426}

\email{nscoville@ursinus.edu}

\date{\today}
\keywords{Discrete Morse theory, Morse complex, dominated vertex, strong collapsibility, homotopy}
\subjclass[2020]{ (Primary) 57Q70, 55U05;  (Secondary) 57Q05, 08A35}

\begin{abstract}
In this paper, we determine the homotopy type of the Morse complex of certain collections of simplicial complexes by studying dominating vertices or strong collapses. We show that if $K$ contains two leaves that share a common vertex, then the Morse complex is strongly collapsible and hence has the homotopy type of a point.  We also show that the pure Morse complex of a tree is strongly collapsible, thereby recovering as a corollary a result of Ayala et al.  In addition, we prove that the Morse complex of a disjoint union $K\sqcup L$ is the Morse complex of the join $K*L$.  This result is used to compute the homotopy type of the Morse complex of some families of graphs, including Caterpillar graphs, as well as the automorphism group of a disjoint union for a large collection of disjoint complexes.
\end{abstract}

\maketitle

\section{Introduction}

Let $K$ be an abstract, finite simplicial complex.  Forman's discrete Morse theory \cite{F-95,F-02} yields a method by which one can construct a gradient vector field on $K$.  This can be used, for example, to compute the Betti numbers of $K$ \cite[Section 8]{F-95}.  The Morse complex of $K$, denoted $\M(K)$, was introduced by Chari and Joswig \cite{CJ-2005} in 2005 as the simplicial complex of all gradient vector fields on $K$. The Morse complex of $K$ is rich enough to reconstruct the isomorphism type of $K$ as was shown by Capitelli and Minian \cite{CM-17}.  Yet, it was shown by the same authors that the simple homotopy type of the Morse complex does not determine the simple homotopy type of $K$.  In general, an explicit computation of the homotopy type of the Morse complex is known for only a handful of complexes. Kozlov determined the homotopy type of the Morse complex of paths and cycles \cite{Kozlov99} while other authors have looked at the complex of discrete Morse functions generated by the maximum gradient vector fields, the so-called pure Morse complex \cite{CJ-2005, A-F-Q-V-08}. While some general connectivity bounds are known \cite{scoville2020higher}, determining the homotopy type of the Morse complex remains elusive, even for a complex as small as the $3$-simplex.

The goal of this paper is to compute the homotopy type of the Morse complex of several classes of simplicial complexes including a general sufficient condition that guarantees the Morse complex is collapsible. One tool we use to aid in our computation is that of a dominating vertex. If every maximal facet of vertex $u$ also contains vertex $v$, then $v$ is said to dominate $u$. In this case, all simplices of $K$ containing $u$ may be removed from $K$ without changing the homotopy type.  Such a removal is called a strong collapse.  In fact, because it can be shown that a strong collapse is a sequence of collapses in Forman's sense, removing or attaching dominated vertices creates a new notion of equivalence known as strong homotopy \cite{BarmakThesis, BarMin12}. If there is a sequence of strong collapses from $K$ to a vertex, then $K$ is called strongly collapsible.  We prove that the existence of two leaves attached to the same vertex on $K$ guarantees that the resulting Morse complex is strongly collapsible, and hence, its Morse complex has the homotopy type of a point. Proposition \ref{prop: pure complex tree sc} shows the pure Morse complex of a tree is strongly collapsible, extending a result of Ayala et al. \cite{A-F-Q-V-08}.

We also show that the Morse complex of a disjoint union of $K$ and $L$ is the join of the Morse complexes on $K$ and $L$.  This result is used several times throughout this paper. If $K=C_n\vee \ell$ where $C_n$ is a cycle of length $n$ and $\ell$ is a leaf, we prove that the Morse complex of $K$ collapses to the Morse complex of a disjoint union of a path of length $n-1$ and path of length $1$. Using the result mentioned above along with Kozlov's result, we then compute the homotopy type of the Morse complex of $C_n\vee \ell$. Furthermore, if $G$ is any graph, we may attach a leaf to each vertex, forming a new graph $L(G)$.  The centipede graph is one type of graph that is obtained in this way. We then show that the Morse complex of $L(G)$ has the homotopy type of $S^{|V(G)|-1}$. Finally, we use the disjoint union result to compute the automorphism group of the Morse complex of the disjoint union for a large collection of complexes, continuing work begun in \cite{LinSco-19}, where a subset of the authors computed the automorphism group of the Morse complex of any connected simplicial complex $K$.

\section{Background}\label{sec: background}
In this section we establish the notation, terminology, and background results that will be needed throughout this paper. All simplicial complexes are assumed to be connected unless otherwise stated. We use $\simeq$ to denote a homotopy equivalence and $\cong$ to denote isomorphic.

\subsection{Simplicial complexes and the Morse complex}

Here we recall some basic notions of simplicial complexes and the Morse complex. Our reference for simplicial complexes is \cite{F-P11} or \cite{J-11} while references for discrete Morse theory and the Morse complex are found in \cite{F-02, KnudsonBook, scoville19}.

%
%

\begin{defn} Let $K$ be a simplicial complex.  A \textbf{discrete vector field} $V$ on $K$ is defined by
$$V:=\{(\sigma^{(p)}, \tau^{(p+1)}) : \sigma< \tau, \text{ each simplex of } K \text{ in at most one pair}\}.
$$
Any pair in $(\sigma,\tau)\in V$ is called a \textbf{regular pair}, and $\sigma, \tau$ are called \textbf{regular simplices} or just \textbf{regular}.  If $(\sigma^{(p)},\tau^{(p+1)})\in V$, we say that $p+1$ is the \textbf{index} of the regular pair. Any simplex in $K$ which is not in $V$ is called \textbf{critical}.
\end{defn}

\begin{defn}
Let $V$ be a discrete vector field on a simplicial complex $K$.  A \textbf{$V$-path} or \textbf{gradient path}  is a sequence of simplices $$\alpha^{(p)}_0, \beta^{(p+1)}_0, \alpha^{(p)}_1, \beta^{(p+1)}_1, \alpha^{(p)}_2\ldots , \beta^{(p+1)}_{k-1}, \alpha^{(p)}_{k}$$ of $K$ such that $(\alpha^{(p)}_i,\beta^{(p+1)}_i)\in V$ and $\beta^{(p+1)}_i>\alpha_{i+1}^{(p)}\neq \alpha_{i}^{(p)}$ for $0\leq i\leq k-1$. If $k\neq 0$, then the $V$-path is called  \textbf{non-trivial.}  A $V$-path is said to be  \textbf{closed} if $\alpha_{k}^{(p)}=\alpha_0^{(p)}$.  A discrete vector field $V$ which contains no  non-trivial closed $V$-paths is called a \textbf{gradient vector field}. We sometimes use $f$ to denote a gradient vector field.
\end{defn}

If the gradient vector field $f$ consists of only a single element, we call $f$ a \textbf{primitive} gradient vector field. Given multiple primitive gradient vector fields, we may combine them to form new gradient vector fields.

If $f,g$ are two gradient vector fields on $K$, write $g\leq f$ whenever the regular pairs of $g$ are also regular pairs of $f$.  In general, we say that a collection of primitive  gradient vector fields $f_0, f_1, \ldots, f_n$ is \textbf{compatible} if there exists a gradient vector field $f$ such that $f_i\leq f$ for all $0\leq i\leq n$.

\begin{defn}\label{MorseComplexDef2}  The  \textbf{Morse complex} of $K$, denoted $\mathcal{M}(K)$, is the simplicial complex whose vertices are given by primitive gradient vector fields and whose $n$-simplices are given by gradient vector fields with $n+1$ regular pairs.  A gradient vector field $f$ is then associated with all primitive gradient vector fields $f:=\{f_0, \ldots, f_n\}$ with $f_i\leq f$ for all $0\leq i\leq n$.
\end{defn}

\begin{ex}\label{complex of P_3} As a simple example, we find the More complex of the following complex $K$:

$$
\begin{tikzpicture}[
    decoration={markings,mark=at position 0.6 with {\arrow{triangle 60}}},
    ]

\node[inner sep=1pt, circle, fill=black] (u) at (0,0) [draw] {};
\node[inner sep=1pt, circle, fill=black] (v) at (3,0) [draw] {};
\node[inner sep=1pt, circle, fill=black] (w) at (6,0) [draw] {};

\draw[-]  (u)--(v);
\draw[-]  (v)--(w);

\node[anchor = north]  at (u) {{$u$}};
\node[anchor = north]  at (v) {{$v$}};
\node[anchor = north]  at (w) {{$w$}};.

\end{tikzpicture}
$$

There are four primitive gradient vector field, namely, $(u,uv), (w,vw), (v,uv)$, and $(v,vw)$ along with compatabilities   $V_1=\{(u,uv), (v,vw)\}, V_2=\{(w,vw), (v,uv)\}$, and $V_3=\{(u,uv), (w,vw)\}$.  Hence the Morse complex is given by

$$
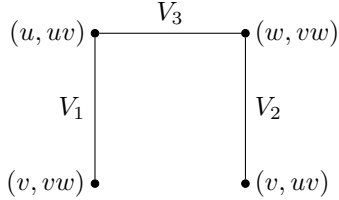
\begin{tikzpicture}[
    decoration={markings,mark=at position 0.6 with {\arrow{triangle 60}}},
    ]

\node[inner sep=1pt, circle, fill=black] (a) at (0,2) [draw] {};
\node[inner sep=1pt, circle, fill=black] (b) at (0,0) [draw] {};
\node[inner sep=1pt, circle, fill=black] (c) at (2,0) [draw] {};
\node[inner sep=1pt, circle, fill=black] (d) at (2,2) [draw] {};

\draw[-]  (a)--(b) node[midway, left] {$V_1$};
\draw[-]  (c)--(d) node[midway, right] {$V_2$};
\draw[-]  (a)--(d) node[midway, above] {$V_3$};

\node[anchor = east]  at (a) {$(u,uv)$};
\node[anchor = east]  at (b) {$(v,vw)$};
\node[anchor = west]  at (c) {$(v,uv)$};
\node[anchor = west]  at (d) {$(w,vw)$};.

\end{tikzpicture}
$$

\end{ex}

\begin{rem} If $(u,vu)$ is a primitive gradient vector field, we sometimes denote this as $(u)v$, and if $(vw, vwu)$ is a primitive gradient vector field, we sometimes denote this as $(vw)u$.
\end{rem}

\subsection{Strong collapsibility}

In this section, we review the basics of dominating vertices and strong collapsibility.  Many of the ideas in this section are originally due to J. Barmak \cite{BarmakThesis,BarMin12}.

\begin{defn} Let $K$ be a simplicial complex.  A vertex $v$ is said to  \textbf{dominate} $v'$ (it is also said that $v'$ is \textbf{dominated} by $v$) if every facet of $v'$ also contains $v$.
\end{defn}

We use the notation $K-\{v'\}:=\{\sigma \in K: v'\not \in \sigma\}$. It is easy to see that if $v'$ is dominated by some vertex $v\in K$, then $K-\{v'\}$ is a simplicial complex.

\begin{defn} If $v$ dominates $v'$, then the removal of $v'$ from $K$ is called an \textbf{elementary strong collapse}  and is denoted by $K\searrow\searrow K-\{v'\}$.  The addition of a dominated vertex is an \textbf{elementary strong expansion}, and is denoted $\nearrow\nearrow$.  A sequence of elementary strong collapses or elementary strong expansions is also called a strong collapse or strong expansion, respectively, and also denoted $\nearrow\nearrow$ or $\searrow\searrow$, respectively. If there is a sequence of strong collapses and expansions from $K$ into $L$, then $K$ and $L$ are said to have the same  \textbf{strong homotopy type}, denoted $K\approx L.$ In particular, if $L=*$, then $K$ is said to have the \textbf{strong homotopy type of a point}. If there is a sequence of elementary strong collapses from $K$ to a point, $K$ is called \textbf{strongly collapsible}.
\end{defn}

\begin{rem}
Since a strong collapse is a sequence of collapses, it follows that if a complex is strongly collapsible, then it is collapsible and hence has the homotopy type of a point.
\end{rem}

Call a simplicial complex $K$  \textbf{minimal} if it contains no dominating vertices.

\begin{ex}\label{ex: argentinean complex} The following simplicial complex is minimal since it has no dominating vertices.

$$
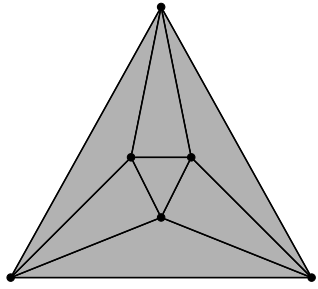
\begin{tikzpicture}[scale = .4]
\filldraw[fill=black!30, draw=black] (0,0)--(10,0)--(5,2)--cycle;
\filldraw[fill=black!30, draw=black] (0,0)--(5,2)--(4,4)--cycle;
\filldraw[fill=black!30, draw=black] (0,0)--(4,4)--(5,9)--cycle;
\filldraw[fill=black!30, draw=black] (10,0)--(5,2)--(6,4)--cycle;
\filldraw[fill=black!30, draw=black] (10,0)--(6,4)--(5,9)--cycle;
\filldraw[fill=black!30, draw=black] (5,2)--(4,4)--(6,4)--cycle;
\filldraw[fill=black!30, draw=black] (4,4)--(6,4)--(5,9)--cycle;

\node[inner sep=1pt, circle, fill=black] (1) at (0,0) [draw] {};
\node[inner sep=1pt, circle, fill=black] (2) at (10,0) [draw] {};
\node[inner sep=1pt, circle, fill=black] (3) at (5,2) [draw] {};
\node[inner sep=1pt, circle, fill=black] (4) at (4,4) [draw] {};
\node[inner sep=1pt, circle, fill=black] (5) at (6,4) [draw] {};
\node[inner sep=1pt, circle, fill=black] (6) at (5,9) [draw] {};

\node[inner sep=0pt](a) at (5,3/4) {};
\node[inner sep=0pt](b) at (3.5,2.3) {};
\node[inner sep=0pt](c) at (5,3.4) {};
\node[inner sep=0pt](d) at (6,2.3) {};
\node[inner sep=0pt](e) at (3.2,4.8) {};
\node[inner sep=0pt](f) at (5,5) {};
\node[inner sep=0pt](g) at (6.6,4.8) {};

\path[style=semithick] (1) edge node[]{\small{}}(2);
\path[style=semithick] (1) edge node[]{\small{}}(6);
\path[style=semithick] (1) edge node[]{\small{}}(3);
\path[style=semithick] (1) edge node[]{\small{}}(4);
\path[style=semithick] (2) edge node[]{\small{}}(3);
\path[style=semithick] (2) edge node[]{\small{}}(5);
\path[style=semithick] (2) edge node[]{\small{}}(6);
\path[style=semithick] (3) edge node[]{\small{}}(4);
\path[style=semithick] (3) edge node[]{\small{}}(5);
\path[style=semithick] (4) edge node[]{\small{}}(5);
\path[style=semithick] (4) edge node[]{\small{}}(6);
\path[style=semithick] (5) edge node[]{\small{}}(6);

\end{tikzpicture}
$$

Note, however, that $K$ is collapsible as well as contractible.
\end{ex}

\begin{defn}\label{defn: core} Let $K$ be a simplicial complex.  The \textbf{core} of $K$ is the minimal subcomplex $K_0 \subseteq K$ such  that $K\searrow\searrow K_0$.
\end{defn}

By \cite[Theorem 5.1.10]{BarmakThesis}, the use of the definite article ``the" in Definition \ref{defn: core} is justified. It follows immediately that the order in which one performs strong collapses on a complex $K$ does not matter, as any sequence of strong collapses of $K$ will eventually yield $K_0$.

One construction that is particularly well-behaved with respect to strong collapses is the join.

\begin{defn}\label{defn: join} Let $K,L$ be two simplicial complexes with no vertices in common.  Define the \textbf{join} of $K$ and $L$, denoted $K*L$, by
$$K*L:= \{\sigma, \tau, \sigma\cup \tau : \sigma \in K, \tau \in L\}.$$ The special case when $L=\{v,w\}$ for vertices $v,w\not \in K$ is the \textbf{suspension} {$\Sigma K$} of $K$.
\end{defn}

If one of the factors in the join is strongly collapsible, then the join is strongly collapsible.

 \begin{prop}\label{prop: Barmakthesis}\cite[Proposition 5.1.16]{BarmakThesis} Let $K, L$ be simplicial complexes.  Then $K*L$ is strongly collapsible if and only if $K$ or $L$ is strongly collapsible.
 \end{prop}

\section{Collapsibility of $\M(K)$}

We study the collapsibility of $\M(K)$ by looking at the slightly more general question of strong collapsibility of $\M(K)$. We first give a general condition on $K$ that guarantees $\M(K)$ is not strongly collapsible in Proposition \ref{prop: no leaf implies not collapsible}.

\subsection{Minimal Morse complexes}

\begin{lem}\label{lem: index 1}
Let $K$ be a simplicial complex. If $(\sigma_1^{(p)},\tau^{(p+1)}) \in V(\mathcal{M}(K))$ dominates some other vertex $(\alpha, \beta) \in \mathcal{M}(K)$, then $p = 0$.
\end{lem}

\begin{proof} Suppose that $p > 0$. Since $\tau$ is of dimension $p+1$, it has exactly $p$ codimension 1 faces (including $\sigma_1$), say $\sigma_1, \sigma_2, \ldots, \sigma_{p+1}.$  These may be paired with $\tau$ to create a primitive vector field $(\sigma_i,\tau)$ on $K$ which in turn corresponds to vertices
 $$(\sigma_1, \tau), (\sigma_2,\tau) ,(\sigma_3,\tau), \dots , (\sigma_{p+1},\tau) \in V(\M(K)).$$

 Notice that $(\sigma_1,\tau)$ is not compatible with any of these vertices. Thus no facet of $(\alpha, \beta)$ contains any of those $p$ vertices so that $(\alpha, \beta)$ must be incompatible with those $p$ vertices. Since $(\alpha,\beta)$ is incompatible with $(\sigma_2, \tau)$, exactly one of the following must occur:
\begin{enumerate}
\item[i)] $\alpha = \sigma_2$
\item[ii)] $\beta = \sigma_2$
\item[iii)] $\alpha = \tau$
\item[iv)] $\beta = \tau$
\end{enumerate}

Now if either iii) or iv) hold, then $(\alpha, \beta)$ will not be compatible with $(\sigma_1, \tau)$. Therefore, either i) or ii) must hold. We proceed by cases.

\textbf{Case 1: } $\alpha = \sigma_2$. Then $(\alpha, \beta) = (\sigma_2, \beta)$ where $\beta \neq \tau$. This implies that $(\alpha, \beta)$ is compatible with $(\sigma_3, \tau)$, a contradiction.

\textbf{Case 2:} $\beta = \sigma_2$. Then $(\alpha, \beta) = (\alpha, \sigma_2)$. We must have $\dim \alpha = \dim\sigma_2 - 1 = p-1$. Therefore $\alpha \neq \sigma_3$, so $(\alpha, \beta)$ is compatible with $(\sigma_3, \tau)$, again a contradiction.

We conclude that no such vertex $(\sigma_1,\tau)$ exists for $p>0$.
\end{proof}

\begin{rem}
Although a dominating vertex in $\M(K)$ cannot come from a vector of index greater than $1$, if $\dim \sigma_1 = 0$, it is possible for $(\sigma_1, \tau)$ to dominate another vertex. In our proof, we required that $\tau$ have at least $3$ faces of codimension $1$, but if $\dim \tau = 1$, there are only two faces of codimension $1$. A simple example is $K=$
$$
\begin{tikzpicture}[
    decoration={markings,mark=at position 0.6 with {\arrow{triangle 60}}},
    ]

\node[inner sep=1pt, circle, fill=black] (u) at (0,0) [draw] {};
\node[inner sep=1pt, circle, fill=black] (v) at (3,0) [draw] {};
\node[inner sep=1pt, circle, fill=black] (w) at (6,0) [draw] {};

\draw[-]  (u)--(v);
\draw[-]  (v)--(w);

\node[anchor = north]  at (u) {\small{$u$}};
\node[anchor = north]  at (v) {\small{$v$}};
\node[anchor = north]  at (w) {\small{$w$}};.

\end{tikzpicture}
$$

We saw in Example \ref{complex of P_3} that the Morse complex is given by:

$$
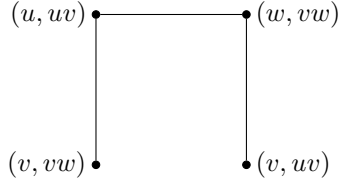
\begin{tikzpicture}[
    decoration={markings,mark=at position 0.6 with {\arrow{triangle 60}}},
    ]

\node[inner sep=1pt, circle, fill=black] (a) at (0,2) [draw] {};
\node[inner sep=1pt, circle, fill=black] (b) at (0,0) [draw] {};
\node[inner sep=1pt, circle, fill=black] (c) at (2,0) [draw] {};
\node[inner sep=1pt, circle, fill=black] (d) at (2,2) [draw] {};

\draw[-]  (a)--(b) node[midway, left] {};
\draw[-]  (c)--(d) node[midway, right] {};
\draw[-]  (a)--(d) node[midway, above] {};

\node[anchor = east]  at (a) {$(u,uv)$};
\node[anchor = east]  at (b) {$(v,vw)$};
\node[anchor = west]  at (c) {$(v,uv)$};
\node[anchor = west]  at (d) {$(w,vw)$};.

\end{tikzpicture}
$$

In this case, vertex $(v,vw)$ dominates $(u,uv)$.
\end{rem}

\begin{prop}\label{prop: no leaf implies not collapsible}
Let $K$ be a simplicial complex. If all vertices $v \in V(K)$ have degree at least $2$, then $\mathcal{M}(K)$ is minimal. In particular, $\M(K)$ is not strongly collapsible.
\end{prop}

\begin{proof}
By Lemma \ref{lem: index 1}, in order to show that no vertex in $\mathcal{M}(K)$ dominates any other, we need only consider vertices in $\M(K)$ which correspond to a primitive vector of index $1$. Hence consider any vertex $(v)a \in V(\mathcal{M}(K))$ where $v,a \in V(K)$. We claim that $(v)a$ cannot dominate any other vertex of $\mathcal{M}(K)$. Suppose  $w \in V(\mathcal{M}(K))$ is any vertex with a facet $\sigma$ that also contains $(v)a$.
Since $v$ has degree at least $2$, there exists a vertex $b\in V(K)$, $b\neq a$, such that $vb$ is a simplex of $K$. This gives rise to the primitive vector $(v)b$ which is also a vertex of $\M(K)$. If $w = (b)v$, then $w$ is compatible with $(a)v$, so $(v)a$ cannot dominate $w$. Now consider $w \neq (b)v$. Then since $w$ is compatible with $(v)a$, it also must be compatible with $(v)b$. Clearly $w \in \sigma - \{(v)a, (v)b\}$ so that there is a facet of $w$ that contains $\sigma - \{(v)a\} \cup \{(v)b\}$ as a face. Then $w$ has a facet that does not contain $(v)a$, so $(v)a$ does not dominate $w$. As this holds for any $(v)a \in V(\mathcal{M}(K))$, no vertex of $\mathcal{M}(K)$ can dominate another vertex. It follows that $\mathcal{M}(K)$ is minimal.
\end{proof}

\begin{cor}
If $\mathcal{M}(K)$ is not minimal, there exists at least one vertex $v \in K$ with degree $1$.
\end{cor}

\subsection{Leaves and strong collapsibility}\label{sec: two leaves}

By Proposition \ref{prop: no leaf implies not collapsible}, if $K$ does not contain a leaf, then $\M(K)$ is minimal (and in particular, not strongly collapsible).  On the other hand, if $K$ has two leaves that share a common vertex, $\M(K)$ is strongly collapsible.

\begin{prop}\label{prop: two leaves strongly collapsible}
If $K$ has two leaves sharing a vertex, then $\M(K)$ is strongly collapsible.
\end{prop}
\begin{proof}

Call the leaves $\{a,ab\}$ and $\{a,ac\}$ where $a,b,c \in V(K)$. These correspond to vertices $(a)b,(b)a,(a)c,(c)a \in V(\M(K))$. We claim that $(b)a$ dominates $(a)c$. Consider any facet $\sigma$ of $(a)c$. The only vertex incompatible with $(b)a$ is $(a)b$, but since $(a)c$ and $(a)b$ are incompatible, $(a)b \not\in \sigma$. Therefore we must have $(b)a \in \sigma$ since $\sigma$ is maximal. Perform the strong collapse given by removing vertex $(a)c$. We claim that $(c)a$ dominates every vertex in the resulting complex. Consider an arbitrary $v \in V(\M(K)) - (a)c$ and a facet $\tau$ containing $v$. The only vertex that $(c)a$ is incompatible with in $V(\M(K))$ is $(a)c$. Since $(a)c \not\in V(\M(K)) - (a)c$, we know that $(c)a$ is compatible with every vertex in $\tau$, so $(c)a \in \tau$. Therefore $(c)a$ dominates $v$. We repeatedly apply the strong collapse removing each vertex $v$, strongly collapsing the Morse complex to $(c)a$.
\end{proof}

Recall that the \textbf{pure Morse complex} of $K$, denoted $\M_P(K)$ is the subcomplex of $\M(K)$ generated by the maximum facets of dimension $\dim(\M(K))$ where a facet $\sigma$ is \textbf{maximum} if $\dim(\sigma)\geq \dim(\tau)$ for every simplex $\tau$, i.e., the complex generated by the maximum gradient vector fields on $K$. A simplex is \textbf{maximal} if it is not contained in any other simplex. Ayala et al. \cite{A-F-Q-V-08} showed that $\M_P(T)$ is collapsible where $T$ is a tree.  We generalize this result by showing that $\M_P(T)$ is strongly collapsible.  The result of Ayala et al. then follows as an immediate corollary.

\begin{prop}\label{prop: pure complex tree sc}
Let $T$ be a tree.  Then $\M_P(T)$ is strongly collapsible.
\end{prop}

\begin{proof}
 Let $T$ be a tree. By definition, $T$ has at least one leaf, say $\{a, ab\}$. We will show that $(b)a$ is dominated and that after removing $(b)a$ from $\M_P(T)$, then $(a)b$ dominates all remaining primitive gradient vector fields, and hence $\M_P(T)$ is strongly collapsible.
	
	    Let $bc$ be an edge incident with $b$, $c\neq a$. We claim that $(b)a$ is dominated by $(c)b$ in $\M_P(T)$.  Suppose $\sigma\in \M_P(T)$ is a facet containing $(b)a$.  Then $\sigma$ is a maximal gradient vector of $T$, and since $\sigma$ is in the pure Morse complex, $\sigma$ is also maximum.  Since $(b)a\in \sigma$ with $a$ a leaf, $\sigma$ is the only maximum gradient vector field containing $(b)a$ and thus is dominated by $(c)b$ (and in fact every primitive gradient vector field of $\sigma$).  Since $(b)a$ is dominated, we may perform a strong elementary collapse and remove it from $\M_P(T)$.
	
	  Now we claim $(a)b$ dominates all remaining primitive gradient vector fields. Note that since $(b)a\not \in \M_P(T)-\{(b)a\}$, $(a)b$ is compatible with  $(\alpha)\beta\in \M_P(T)-\{(b)a\}$.   Thus $(a)b$ dominates $(\alpha)\beta$ for all other $(\alpha)\beta\in \M_P(T)-\{(b)a\}$ so that $\M_P(T)-\{(b)a\}$ is a cone and thus strongly collapsible.
\end{proof}

We then recover the result mentioned above.

\begin{cor} Let $T$ be a tree.  Then $\M_P(T)$ is collapsible.
\end{cor}

Let $P_t$ denote the path consisting of $t+1$ vertices $v_0, v_1, ..., v_{t}$ with $v_{i}$ adjacent to $v_{i+1}$ for $0 \leq i \leq t-1$. We slightly strengthen Kozlov's computation \cite[p. 119]{Kozlov99} by showing that a path on $3n$ vertices is strongly collapsible. First a lemma that will prove useful.

	\begin{lem} \label{lem: dominate}
	    Let $K$ be a simplicial complex with leaf $\{a, ab\}$ and $c$ a neighbor of $b$ not equal to $a$. Then $(b)c$ is dominated in $\mathcal{M}(K)$ by $(a)b$.
	\end{lem}
	
	\begin{proof}
	Consider any facet of $(b)c$ in $\mathcal{M}(K)$. A facet of $\mathcal{M}(K)$ is a maximal gradient vector field on $K$, and since $(b)a$ is not compatible with $(b)c$ and $\{a,ab\}$ is a leaf, $(a)b$ must be in any maximal gradient vector field containing $(b)c$. Thus $(a)b$ dominates $(b)c$ in $\mathcal{M}(K)$.
	\end{proof}

\begin{prop}\label{prop: path strong collapse} Let $P_{3n-1}$ be the path on $3n$ vertices, $n\geq 1$.
    Then $\M(P_t)\searrow\searrow *$
\end{prop}

\begin{proof}
    By Lemma \ref{lem: dominate},  $(v_1)v_2$ dominates $(v_2)v_3$. After removing $(v_2)v_3$, we see that  $(v_3)v_2$ dominates $(v_4)v_3$, and so we remove $(v_4)v_3$. Continuing in this manner, we see that $(v_{3k-2})v_{3k-1}$ dominates $(v_{3k-1})v_{3k}$ for all $1\leq k \leq n$, and $(v_{3k})v_{3k-1}$ dominates $(v_{3k+1})v_{3k}$ for for all $1\leq k <n$. Hence we may remove each of these primitive gradient vector fields.

    Now the last primitive gradient vector field removed is $(v_{3n-1})v_{3n}$  since it was dominated by $(v_{3n-2})v_{3n-1}$. We now claim that $(v_{3n})v_{3n-1}$  dominates every remaining vertex. To see this, observe that because $(v_{3n-1})v_{3n}$ has been removed, $(v_{3n})v_{3n-1}$ is compatible with all remaining vertices $(v_i)v_j$, and no $(v_i)v_j$ can exist in a facet of the remaining Morse complex without $(v_{3n})v_{3n-1}$. We remove all $(v_i)v_j$ until we are only left with $(v_{3n})v_{3n-1}$. Thus $\mathcal{M}(P_{3n-1})$ is strongly collapsible.
\end{proof}

\section{Morse complex of the disjoint union}\label{sec: Morse complex of the join}

Before using strong collapses to compute the homotopy type of the Morse complex of some other families of graphs, we need a result, interesting in its own right, about the Morse complex of a disjoint union. This result will be used in Section \ref{sec: Morse complex of some families of graphs} as well as Section \ref{sec: Automorphism group of the Morse complex of a disconnected complex} where we investigate the automorphism group of the Morse complex of a disjoint union.

\begin{lem}\label{lem: subset Morse}
If $A\subseteq B$, then $\M(A)\subseteq \M(B)$.
\end{lem}

\begin{proof}
Consider any primitive pair $(\sigma, \tau) \in V(\M(A))$ where $\sigma,\tau \in A$. Then we have $\sigma,\tau\in B$, thus $(\sigma, \tau) \in V(\M(B))$, so $V(\M(A)) \subseteq V(\M(B))$. Now consider any simplex $\sigma = \sigma_1\sigma_2\cdots \sigma_m \in \M(A)$. Since all of the vertices $\sigma_1, \sigma_2, \dots , \sigma_m$ are compatible in $\M(A)$ and are vertices in $\M(B)$, they must also be compatible in $\M(B)$. Therefore $\sigma \in \M(B)$.
\end{proof}

\begin{prop}\label{prop: morse union}
Let $K,L$ be connected simplicial complexes each with at least one edge.  Then $\M(K\sqcup L)= \M(K)\ast \M(L)$.
\end{prop}

\begin{proof}
Assume without loss of generality that $K$ and $L$ are disjoint.  We first claim that $V(\M(K\cup L)) =  V(\M(K)\ast \M(L))$. Notice $V(\M(K)\ast \M(L)) = V(\M(K)) \cup V(\M(L))$ since the join operation does not create or remove any vertices. Consider any pair $(\sigma,\tau) \in V(\M(K\cup L))$. Then since $K$ and $L$ are disjoint, we have $\sigma,\tau \in K$ or $\sigma, \tau \in L$. Therefore $(\sigma,\tau) \in V(\M(K))$ or $(\sigma,\tau)\in V(\M(L))$, so $(\sigma,\tau) \in V(\M(K))\cup V(\M(L))$. Thus $V(\M(K\cup L))\subseteq V(\M(K)) \cup V(\M(L))$.\\

Now consider any $(\alpha,\beta) \in V(\M(K)\ast \M(L))$. Then $(\alpha,\beta) \in V(\M(K)) $ or $(\alpha,\beta) \in V(\M(L))$. Without loss of generality suppose $(\alpha,\beta) \in V(\M(K)) $. By Lemma \ref{lem: subset Morse}, we have $(\alpha,\beta) \in V(\M(K\cup L))$. Thus $V(\M(K)) \cup V(\M(L))\subseteq V(\M(K\cup L))$, so that $V(\M(K)) \cup V(\M(L))= V(\M(K\cup L))$. \\

To show that $\M(K \cup L) = \M(K) \ast \M(L)$, consider any simplex $\sigma \in \M(K \cup L)$. Write $\sigma = \alpha\cup\beta$, where $\alpha = \alpha_1\alpha_2\cdots \alpha_a$ and $\beta = \beta_1\beta_2\cdots \beta_b$, in which each $\alpha_i = (\sigma_i, \tau_i)$ where $\sigma_i,\tau_i \in K$ for  $i=1, 2 \dots a$, and $\beta_j = (\gamma_j, \delta_j)$ where $\gamma_j,\delta_j \in L$ for $j=1, 2, \dots b$. Notice that $\alpha\in \M(K \cup L)$. Thus $\alpha$ is also a gradient vector field of $K\cup L$. Moreover, since all $\alpha_i\in \alpha$ are pairs of simplices of $K$, this gradient vector field consists solely of primitive gradient vector fields in $K$. Thus $\alpha\in \M(K)$. By the same reasoning, $\beta \in \M(L)$. Thus it follows that $\sigma = \alpha \cup \beta \in \M(K)\ast \M(L)$. Hence $\M(K \cup L) \subseteq\M(K) \ast \M(L).$\\

Now suppose that $\tau \in \M(K) \ast \M(L)$. Since $\tau \in \M(K) \ast \M(L)$, write $\tau = a \cup b$, for simplices $a \in \M(K)$ and $b \in \M(L)$. This implies that all vertices in $a$ are compatible with each other, and similarly for $b$. Since $K \cap L = \emptyset$, we know that in $M(K\cup L)$, every vertex in $V(K)$ is compatible with every vertex in $V(L)$. It follows that all vertices in $a\cup b$ are pairwise compatible in $M(K \cup L)$. It remains to show that $a\cup b$ does not correspond to a cyclic matching of the induced directed Hasse diagram $\HH(K\cup L)$. Since $K \cap L = \emptyset$, any cycle in $\HH(K\cup L)$ must be contained entirely in its subgraphs $\HH(K)$ or $\HH(L)$. This would imply at least one of $a$ or $b$ corresponds to a cyclic matching of $\HH(K)$ or $\HH(L)$, respectively. However, since $a \in \M(K)$ and $b \in \M(L)$, we know that is not the case. We conclude that $a\cup b$ corresponds to an acyclic matching of $\HH(K\cup L)$, thus $a\cup b \in \M(K\cup L)$. Hence $ \M(K) \ast \M(L)\subseteq \M(K \cup L) $. We conclude that $\M(K \cup L) = \M(K) \ast \M(L)$.
\end{proof}

\begin{cor}\label{cor: suspension} Let $K$ be a simplicial complex.  Then $\M(K\sqcup P_1)\simeq \Sigma\M(K).$
\end{cor}

\begin{example}
While the collection of Morse complexes is closed under joins, not every join is realized as a Morse complex. For example, let $K:=\{a,b,c,ab,bc\}$ and $L:=\{u,v,uv\}$ so that the join $K*L$ is given by

$$
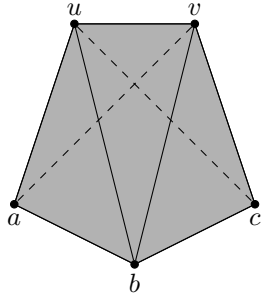
\begin{tikzpicture}[scale=.8,
    decoration={markings,mark=at position 0.6 with {\arrow{triangle 60}}},
    ]

\filldraw[fill=black!30, draw=black] (0,-1)--(2,-2)--(4,-1)--(3,2)--(1,2)--cycle;

\node[inner sep=1pt, circle, fill=black] (a) at (0,-1) [draw] {};
\node[inner sep=1pt, circle, fill=black] (b) at (2,-2) [draw] {};
\node[inner sep=1pt, circle, fill=black] (c) at (4,-1) [draw] {};
\node[inner sep=1pt, circle, fill=black] (u) at (1,2) [draw] {};
\node[inner sep=1pt, circle, fill=black] (v) at (3,2) [draw] {};

\draw[-]  (a)--(b) node[midway, below] {};
\draw[-]  (b)--(c) node[midway, left] {};
\draw[-]  (u)--(v) node[midway, below] {};
\draw[-]  (u)--(a) node[midway, left] {};
\draw[-]  (u)--(b) node[midway, right] {};
\draw[-, dashed]  (u)--(c) node[midway, below] {};
\draw[-, dashed]  (v)--(a) node[midway, above] {};
\draw[-]  (v)--(b) node[midway, above] {};
\draw[-]  (v)--(c) node[midway, above] {};

\node[anchor = north ]  at (a) {{$a$}};
\node[anchor = north ]  at (b) {{$b$}};
\node[anchor = north ]  at (c) {{$c$}};
\node[anchor = south]  at (u) {{$u$}};
\node[anchor = south ]  at (v) {{$v$}};

\end{tikzpicture}
$$

Suppose $K*L=\M(N)$ for some simplicial complex $N$. If $N$ contains a $2$-simplex, then there are at least $9$ primitive gradient vector fields on $N$, hence at least $9$ vertices in $\M(N)$, a contradiction.  Hence $N$ must be a graph. But the number of primitive gradient vector fields on a graph is even, again a contradiction.  Thus $K*L$ is not the Morse complex of any simplicial complex.
\end{example}

In addition, there are simplicial complexes $K$ such that $\Sigma K \neq \M(L)$ for any $L$. A similar argument to the one above shows that $\Sigma \Delta^2\neq \M(K)$ for any simplicial complex $K$.

\section{Morse complex of some families of graphs}\label{sec: Morse complex of some families of graphs}

Our main goal in this section is to compute the homotopy type of the Morse complex of several families of graphs by strongly collapsing the Morse complex to the Morse complex of a disjoint union and applying Proposition \ref{prop: morse union}. This is in part accomplished through an interesting observation concerning the role of strong collapses in certain Morse complexes.  We begin with an example.

\subsection{The Morse complex of cycles wedge a leaf}\label{sec: The Morse complex of cycles wedge a leaf}

\begin{ex}\label{ex: cicle collapse}
Let $C_3$ be the cycle on three vertices
$$
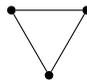
\begin{tikzpicture}[decoration={markings,mark=at position 0.6 with {\arrow{triangle 60}}},baseline]

\node[inner sep=1pt, circle, fill=black] (a) at (0,0.866) [draw] {};
\node[inner sep=1pt, circle, fill=black] (b) at (1,0.866) [draw] {};
\node[inner sep=1pt, circle, fill=black] (c) at (0.5,0) [draw] {};
\draw[-] (a)--(b)--(c)--(a)--cycle;
\end{tikzpicture}
$$
Since $C_3$ contains no leaves, $\M(C_3)$ is not strongly collapsible.  However, if we attach a single leaf to any vertex, it can be shown that the resulting Morse complex strongly collapses to the Morse complex of the disjoint union of the path of length one and the path of length two which strongly collapses to a point.  Evidently, the sequence of strong collapses can be written as

$$\mathcal{M}\left(
\begin{tikzpicture}[decoration={markings,mark=at position 0.6 with {\arrow{triangle 60}}},baseline]

\node[inner sep=1pt, circle, fill=black] (a) at (0,0.866) [draw] {};
\node[inner sep=1pt, circle, fill=black] (b) at (1,0.866) [draw] {};
\node[inner sep=1pt, circle, fill=black] (c) at (0.5,0) [draw] {};
\node[inner sep=1pt, circle, fill=black] (d) at (0.5,-0.866) [draw] {};
\draw[-] (a)--(b)--(c)--(a)--cycle;
\draw[-] (c)--(d);
\end{tikzpicture}
\right)
\searrow\searrow
\mathcal{M}\left(
\begin{tikzpicture}[decoration={markings,mark=at position 0.6 with {\arrow{triangle 60}}},baseline=0.5cm]
\node[inner sep=1pt, circle, fill=black] (a) at (0,1) [draw] {};
\node[inner sep=1pt, circle, fill=black] (b) at (1,1) [draw] {};
\node[inner sep=1pt, circle, fill=black] (c) at (2,1) [draw] {};
\node[inner sep=1pt, circle, fill=black] (d) at (0.5,0) [draw] {};
\node[inner sep=1pt, circle, fill=black] (e) at (1.5,0) [draw] {};
\draw[-] (a)--(b)--(c);
\draw[-] (d)--(e);
\end{tikzpicture}
\right)
\searrow\searrow
\begin{tikzpicture}[decoration={markings,mark=at position 0.6 with {\arrow{triangle 60}}},baseline=0.5cm]
\node[inner sep=1pt, circle, fill=black] (a) at (2,0) [draw] {};
\end{tikzpicture}
$$


\end{ex}

We will prove in Proposition \ref{prop: cycle collapse} that the Morse complex of a cycle wedged with a leaf strongly collapses to the Morse complex of a disjoint union of paths. This, along with other results, will be used in Theorem \ref{thm: morse of cycle and leaf} to prove that

    \begin{gather*}
       \mathcal{M}(C_n\vee \ell) \simeq  {\begin{cases}
                                * &\text{if } n=3k \\
                                S^{2k}  &\text{if } n = 3k+1 \\
                                S^{2k+1} &\text{if } n=3k+2.\\
                                \end{cases}
                                }
    \end{gather*}
where $\ell$ is a path of length 1.

\begin{defn}
Let $\mathbb{P}$ be the set of all (finite) posets, and $\mathbb{K}$ be the set of all simplicial complexes. Define a function $f\colon \mathbb{P} \to \mathbb{K}$ as follows: for each $P \in \mathbb{P}$, construct a simplicial complex $f(P)$ whose vertex set is the edge set of $P$. Then let $\sigma = e_1e_2 \cdots e_k$ be a simplex of $f(P)$ if and only if the edges $e_1, e_2, \cdots e_k$ oriented upward and all other edges oriented downward form an acyclic matching of $P$.
\end{defn}

\begin{rem}\label{rem: general Hasse} Note that for any simplicial complex $K$, $\M(K) \simeq f(\HH(K))$. Our definition thus generalizes the notion of taking the Morse complex to degenerate Hasse diagrams. We will similarly call $f(P)$ the Morse complex of the poset $P$.
\end{rem}

Given Remark \ref{rem: general Hasse} and Proposition \ref{prop: morse union}, we also have the following.

\begin{cor}\label{cor: f join}  Let $A,B$ be posets.  Then $f(\HH(A)\sqcup \HH(B)) \simeq f(\HH(A))\ast f(\HH(B))$.
\end{cor}

It turns out that determining the behaviour of a strong collapse of $\M(K)$ is best seen by studying a modified version of the Hasse diagram of $K$.  In general, this is simply a poset, and not itself a Hasse diagram. Nevertheless, its study will allow us to say something about strong collapses in $\M(K)$.  We use $K\vee_v \ell$ to denote attaching a leaf $\ell$ to a vertex $v\in K$.  We use $K\vee \ell$ when there is no need to make reference to the vertex.

\begin{lem}\label{lem: hasse collapse1}
For any simplicial complex $K$ and vertex $v \in V(K)$, the Morse complex $\M(K \vee_v \ell)$ strongly collapses to $f((\HH(K)-v) \sqcup \HH(\ell))$.
\end{lem}

\begin{proof}
Write $\ell=vw$ for some vertex $w$ and let  $a_1, a_2, \dots , a_k$ be the neighbors of $v$. By Lemma \ref{lem: dominate},  the vertex $(w,wv)$ dominates vertices $(v,va_1), (v,va_2), \dots (v,va_k)$, leading to $k$ strong collapses. In the Hasse diagram $\HH(K \vee \ell)$, this corresponds to a removal of the edges connecting node $v$ to nodes $va_1,va_2, \dots , va_k$. As these are all the edges in $K$ that $v$ is connected to, the Hasse diagram now consists of $\HH(K)$ with node $v$ removed, together with a second component consisting of the Hasse diagram of the leaf $vw$. The entire Hasse diagram is $(\HH(K)-v) \sqcup \HH(\ell)$. Therefore $\M(K \vee_v \ell) \searrow\searrow f((\HH(K)-v) \sqcup \HH(\ell))$.
\end{proof}

 Let $\pd^{n}$ be the boundary of the $n$-simplex on the vertices $\{v_0, v_1, \ldots, v_n\}$ and write $\delta:=v_0v_1\cdots v_{n}$.  Define the \textbf{reflection map} \cite{LinSco-19} $\pi_n=\pi\colon \pd^n \to \pd^n$ by $\pi(\sigma) := \delta - \sigma $. Note that the reflection map is not a simplicial map.

\begin{prop}\label{prop: morse collapse sphere} Let $v$ be a vertex of $\pd^n$. Then $\M(\pd^n \vee_v \ell) \searrow \searrow \M( (\pd^n - \pi(v))\sqcup \ell)$.
\end{prop}

\begin{proof}
By Lemma \ref{lem: hasse collapse1},  $\M(\pd^n \vee_v \ell) \searrow \searrow f((\HH(\pd^n) - v) \sqcup \HH(\ell))$. By Corollary \ref{cor: f join}, we have $f((\HH(\pd^n) - v) \sqcup \HH(\ell)) \simeq f(\HH(\pd^n)-v) \ast f(\HH(\ell))$, and since $f(\HH(\ell))=\M(\ell)$ by Remark \ref{rem: general Hasse}, $f((\HH(\pd^n) - v) \sqcup \HH(\ell)) \simeq f(\HH(\pd^n)-v) \ast \M(\ell)$. The same argument shows that $\M( (\pd^n - \pi(v))\sqcup \ell) \searrow\searrow f(\HH( \pd^n) - \pi(v)) \ast \M(\ell)$. It thus suffices to show that $f(\HH(\pd^n)-v) \simeq f(\HH( \pd^n) - \pi(v))$.
Now $\pi_n$ is a bijection so that the Hasse diagram $\pi_n(\HH(\pd^n)-v)$ will be the Hasse diagram $\HH(\pd^n)$ with $\pi(v)$ removed. This is precisely $\HH(\pd^n)  - \pi(v)$. Therefore $\HH(\pd^n)-v \simeq \HH(\pd^n) - \pi(v)$, so $f(\HH(\pd^n)-v) \simeq f(\HH( \pd^n) - \pi(v))$.
\end{proof}

Recall that in Example \ref{ex: cicle collapse}, we saw that $\M(C_3\vee \ell)$ strongly collapses to the disjoint union of a two paths.  Equipped with Lemma \ref{lem: hasse collapse1}, we now show that this occurs for a cycle of any length.

\begin{prop}\label{prop: cycle collapse} Let $v$ be a vertex of $C_n$. Then  $\M(C_n \vee_v \ell) \searrow\searrow \M(P_{n-1} \sqcup \ell)$.
\end{prop}

\begin{proof}
We follow a similar method to that of the proof of Proposition \ref{prop: morse collapse sphere}.  By Lemma \ref{lem: hasse collapse1}, we know that $\M(C_n \vee_v \ell) \searrow\searrow f((\HH(C_n)-v)\sqcup \HH(\ell))$.  We also have $f((\HH(C_n)-v)\sqcup \HH(\ell)) \simeq f(\HH(C_n)-v) \ast f(\HH(\ell)) \simeq f(\HH(C_n)-v)\ast \M(\ell)$. In addition, by Proposition \ref{prop: morse union} and Remark \ref{rem: general Hasse}, $\M(P_{n-1}\sqcup \ell) \simeq \M(P_{n-1})\ast \M(\ell)\simeq f(\HH(P_{n-1}))\ast \M(\ell)$. Observe that $\HH(C_n)-v\simeq \HH(P_{n-1})$, thus $f(\HH(P_{n-1})) \simeq f(\HH(C_n)-v)$.
\end{proof}

This next result is due to Kozlov.

\begin{prop}\label{prop: Kozolv}\cite{Kozlov99}
Let $P_{n-1}$ be a path on $n$ vertices.  Then

    \begin{gather*}
       \mathcal{M}(P_{n-1}) \simeq  {\begin{cases}
                                * &\text{if } n=3k \\
                                S^{2k-1}  &\text{if } n = 3k+1 \\
                                S^{2k} &\text{if } n=3k+2.\\
                                \end{cases}
                                }
    \end{gather*}

\end{prop}

Combining Propositions \ref{prop: cycle collapse}, \ref{prop: morse union}, \ref{prop: Kozolv}, \ref{prop: Barmakthesis}, and Corollary \ref{cor: suspension} yield the following:

\begin{thm}\label{thm: morse of cycle and leaf} Let $C_n$ be a cycle of length $n\geq 3$.  Then
    \begin{gather*}
       \mathcal{M}(C_n\vee \ell) \simeq  {\begin{cases}
                                * &\text{if } n=3k \\
                                S^{2k}  &\text{if } n = 3k+1 \\
                                S^{2k+1} &\text{if } n=3k+2.\\
                                \end{cases}
                                }
    \end{gather*}
\end{thm}

\subsection{The Morse complex of centipede graphs}

Here we will use Proposition \ref{prop: morse union} to compute the homotopy type of the Morse complex of any graph with the property that every vertex is either a leaf or adjacent to exactly one leaf. Centipede graphs satisfy this property.

	\begin{prop}
	    Let $G$ be a connected graph with $v$ vertices, and let $L(G)$ be the complex resulting from adding a leaf to each vertex of $G$ vertex. Then $\mathcal{M}(L(G)) \simeq S^{v-1}$.
	\end{prop}
	
	\begin{proof} Let $G$ be a connected graph, and a call the leaves we add to $G$ to obtain $L(G)$ by $\ell_1, \ell_2, \ldots, \ell_v$. Let $\{a,ab\}$ be any leaf of $G$.  If $c$ is any neighbor of $b$, $c\neq a$, then $(b)c$ is dominated in $\mathcal{M}(L(G))$ by Lemma \ref{lem: dominate}. By adding a leaf to each vertex of $G$,every primitive gradient vector field on $G$ is dominated, and thus can be removed from $\mathcal{M}(L(G))$.  Since the Morse complex of a single leaf is $S^0$, we have
\begin{eqnarray*}
	    \mathcal{M}(\ell_1 \sqcup \ell_2 \sqcup\ldots \sqcup \ell_v) &=& \mathcal{M}(\ell_1) * \mathcal{M}(\ell_2) * \ldots \mathcal{M}(\ell_v)\\
	    &\simeq& S^{v-1}
\end{eqnarray*}
where the first equality is Proposition \ref{prop: morse union} and the second follows from the fact that  $\Sigma S^n \simeq S^{n+1}$.
	\end{proof}

A \textbf{\textit{centipede graph}}, denoted $\mathcal{C}_v$ is a graph obtained by attaching a leaf to each vertex of a path with $v$ vertices. The following corollary is then immediate.

	\begin{cor}\label{cor: cent}
	    Let $\mathcal{C}_v$ be a centipede graph. Then $\mathcal{M}(\mathcal{C}_v) \simeq S^{v-1}$.
	\end{cor}
	
See Figure \ref{fig:M8} for an illustration of Corollary \ref{cor: cent}.

	\begin{center}
		\begin{figure} [h]
		
		\begin{tikzpicture}[scale=1.5, node distance = {20mm}, thick, main/.style = {draw, circle}]
		
		    \node[inner sep=1pt, circle, fill=black](v1) at (-1,0) {};
		    \node[inner sep=1pt, circle, fill=black](v1) at (0,0) {};
            \node[inner sep=1pt, circle, fill=black](v2) at (1,0) {};
            \node[inner sep=1pt, circle, fill=black](v3) at (-1,1) {};
            \node[inner sep=1pt, circle, fill=black](v3) at (0,1) {};
		`   \node[inner sep=1pt, circle, fill=black](v3) at (1,1) {};
		
		    \draw (-1,0) -- (0,0);
		    \draw (0,0) -- (1,0);
		    \draw (-1,0) -- (-1,1);
		    \draw (0,0) -- (0,1);
		    \draw (1,0) -- (1,1);
		
		    \node[inner sep=1pt, circle, fill=black](v1) at (3,0) {};
		    \node[inner sep=1pt, circle, fill=black](v1) at (4,0) {};
            \node[inner sep=1pt, circle, fill=black](v2) at (5,0) {};
            \node[inner sep=1pt, circle, fill=black](v3) at (3,1) {};
            \node[inner sep=1pt, circle, fill=black](v3) at (4,1) {};
		    \node[inner sep=1pt, circle, fill=black](v3) at (5,1) {};
		
		    \draw (3,0) -- (3,1);
		    \draw (4,0) -- (4,1);
		    \draw (5,0) -- (5,1);
		
		\end{tikzpicture}
		
		\caption{On the left, we start with $\mathcal{C}_2$. After performing the strong collapses that Lemma \ref{lem: dominate} allows, we then only have to take the Morse complex of the subcomplex on the right.  }\label{fig:M8}
		\end{figure}
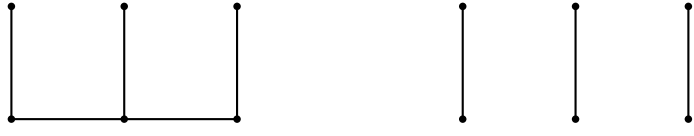
		\end{center}

\subsection{Morse complex of a path with a leaf}

    \begin{lem}\label{lem: leaf in a path}
        Let $v_k$ be a vertex of $P_t$, $1 \leq k \leq t-1$ and $t \geq 2$. Then $\mathcal{M}(P_t \vee_{v_k} \ell) \searrow \searrow \mathcal{M}(P_{k+1} \sqcup P_{t-(k+2)} \sqcup \ell)$.
    \end{lem}

    \begin{proof}
        Write $\ell = v_ku$.  By Lemma \ref{lem: dominate}, $(u)v_k$ dominates  $(v_k)v_{k+1}$ in $\mathcal{M}(P_t \vee_{v_k} \ell)$. In the corresponding Hasse diagram $\mathcal{H}(P_t \vee_{v_k} \ell)$, this corresponds to the removal of the edge between $v_k$ and $v_kv_{k+1}$.

        Furthermore, by Lemma \ref{lem: dominate}, $(v_t)v_{t-1}$ dominates $(v_{t-1})v_{t-2}$. This corresponds to the removal of the edge between $v_{t-1}$ and $v_{t-2}v_{t-1}$ on the Hasse diagram. Upon inspection, we see that this yields three components: the Hasse diagram of the path $P_{k+1}$, the Hasse diagram of the path $P_1=\ell$, and an ``upside-down" Hasse diagram of the path $\mathcal{H}(P_{t-(k+2)})$. Thus
        \begin{align*}
        \mathcal{M}(P_t \vee_{v_k} \ell) \searrow \searrow f(\mathcal{H}(P_{k+1}) \sqcup \mathcal{H}(P_{t-(k+2)}) \sqcup \mathcal{H}(\ell)).
        \end{align*}
        By Proposition \ref{prop: morse union} and Remark \ref{rem: general Hasse}, we have that
        \begin{align*}
        f(\mathcal{H}(P_{k+1}) \sqcup \mathcal{H}(P_{t-(k+2)}) \sqcup \mathcal{H}(\ell)) \simeq \mathcal{M}(P_{k+1} \sqcup P_{t-(k+2)} \sqcup \ell).
        \end{align*}
    \end{proof}

    Combining Lemma \ref{lem: leaf in a path} and Proposition \ref{prop: morse union}, we have

    \begin{prop}\label{prop: path leaf}
        Let $v_k$ be a vertex of $P_t$, $1 \leq k \leq t-1$. Then $\mathcal{M}(P_t \vee_{v_k} \ell) \simeq \mathcal{M}(P_{k+1}) * \mathcal{M}(P_{t-(k+2)}) * \mathcal{M}(\ell)$.
    \end{prop}

    Considering Proposition \ref{prop: path leaf}, Proposition \ref{prop: Barmakthesis}, and Proposition \ref{prop: path strong collapse}, we can conclude the following:

    \begin{cor}
        Let $v_k$ be a vertex of $P_t$,  $1 \leq k \leq t-1$. If $k+1 = 2j$ or $t-(k+2) = 2j$, then $\mathcal{M}(P_t \vee_{v_k} l) \searrow\searrow *$.
    \end{cor}

\section{Automorphism group of the Morse complex of a disconnected complex}\label{sec: Automorphism group of the Morse complex of a disconnected complex}

In this section, we give another application of Proposition \ref{prop: morse union}. In \cite{LinSco-19}, a subset of the authors computed the automorphism group of the Morse complex of any connected simplicial complex.  There it was shown that if $K$ is a connected simplicial complex, then

$$
\Aut(\M(K))\cong
\begin{cases}
\Aut(K) & \text{if } K\neq \pd^n,C_n \\
\Aut(C_{2n}) & \text{if } K= C_n \\
\Aut(K)\times \mathbb{Z}_2 & \text{if } K= \pd^n. \\
\end{cases}
$$

Proposition \ref{prop: morse union} along with the results in this section allow us to compute the automorphism group of the Morse complex for certain disconnected complexes.

\begin{defn} Let $K$ be a simplicial complex. A subcomplex $U\leq K$ is called \textbf{fully connected} in $K$ if $K\cong U \ast(K-U).$
\end{defn}

\begin{example} If $K=\Sigma L=\{u,v\}*L$ for some complex $L$, then the subcomplex $U=\{u,v\}$ of $K$ is fully connected.
\end{example}

\begin{prop}\label{prop: aut product} Let $K,L$ be simplicial complexes.  Then $\Aut(K\ast L) \cong\Aut(K) \times \Aut(L)$ except when there exists subcomplexes $U_1 \leq K, U_2 \leq L$ with $U_1 \cong U_2$, such that $U_1$ is fully connected in $K$ and $U_2$ is fully connected in $L$.
\end{prop}

For example, Proposition \ref{prop: aut product} does not hold for $K=\Sigma K_0$ and $L=\Sigma L_0$.

\begin{proof}
We first show that $\Aut(K) \times \Aut(L)$ is a subgroup of $\Aut(K\ast L)$. Define $\phi\colon \Aut(K) \times \Aut(L)\to \Aut(K\ast L)$ as follows. Let $(a,b) \in \Aut(K) \times \Aut(L)$.  For any $\sigma\in K\ast L$, write $\sigma=\alpha \beta$ for $\alpha \in K$ and $\beta \in L$. Let $\phi_{(a,b)}\colon K*L\to K*L$ by $\phi_{(a,b)} (\sigma) = a(\alpha)b(\beta)$, and define $\phi(a,b)=\phi_{(a,b)}.$  Consider any simplex $\sigma = \alpha\beta \in K \ast L$, with $\alpha\in K$ and $\beta \in L$, and any $(a,b), (c,d) \in \Aut(K)\times \Aut(L)$. We have
\begin{align*}
    (\phi(a,b)\circ\phi(c,d))(\sigma)&= \phi(a,b)(\phi(c,d)(\alpha\beta))\\
    &= \phi(a,b)(c(\alpha)b(\beta))\\
    &= (a\circ c)(\alpha)(b\circ d)(\beta)\\
    &= \phi(a\circ c, b\circ d)(\sigma)\\
    &= \phi((a,b)\circ (c,d))(\sigma).
\end{align*}

Hence $\phi$ is a homomorphism. We now show that $\phi$ is injective. Consider any $(f,g) \in \Ker(\phi)$. Then $\phi(f,g) = \id_{K\ast L}$. Consider $\sigma = \alpha\beta \in K\ast L$, where $\alpha\in K, \beta \in L$. Then we have
\begin{align*}
    \phi(f,g)(\sigma) &=\sigma\\
    f(\alpha)g(\beta)&= \alpha\beta.
\end{align*}
It follows that $f(\alpha)=\alpha$, and $g(\beta) = \beta$. This holds for any choice of $\alpha \in K, \beta \in L$, so $f = \id_K$ and $g = \id_L$. Hence $\Ker(\phi)$ is trivial, so $\phi$ is injective. Therefore $\Aut(K) \times \Aut(L)$ is a subgroup of $\Aut(K\ast L)$. \\

We now show that if the proposed conditions hold, $|\Aut(K\ast L)| > |\Aut(K) \times \Aut(L)|$. In particular, we show that there exists an automorphism in $\Aut(K\ast L)$ that does not emerge from ``combining" automorphisms in $\Aut(K)$ and $\Aut(L)$. We build this from the hypothesis that $U_1\cong U_2$.  Hence let $g\colon U_1 \to U_2$ be an isomorphism. Construct a function $f_V: V(K\ast L) \to V(K\ast L)$ such that $f_V(v) = g(v)$ if $v \in U_1$, $f_V(v) = g^{-1}(v)$ if $v \in U_2$, and $f_V(v) = v$ if $v \not\in U_1 \cup U_2$. It is easy to see that $f_V$ is a bijection. Let $f\colon K\ast L \to K\ast L$ be the induced function on the join.

We first show that $f$ is a simplicial map. Consider any simplices $u_1 \in U_1, u_2 \in U_2, k \in K - U_1, \ell \in L-U_2$. Since $U_1$ is fully connected in $K$ and $U_2$ is fully connected in $L$, we have that each of $U_1, U_2, K-U_1, L-U_2$ is joined to each other (excluding itself) in $K*L$.  By definition, the simplex $u_1u_2k\ell \in K\ast L$. Consider any simplex $\sigma \in K\ast L$. Write $\sigma$ as some $u_1u_2k\ell$ as above, in which $u_1 \in U_1, u_2 \in U_2, k \in K - U_1, \ell \in L-U_2$, allowing any of $u_1, u_2, k, \ell$ to be empty.  Then $f(u_1u_2k\ell) = f(u_1)f(u_2)f(k)f(\ell) = f(u_1)f(u_2)k\ell$. Since $f(u_1) = g(u_1) \in U_2$ and $f(u_2) =g^{-1}(u_2)\in U_1$, we know that $f(\sigma) = g(u_1)g^{-1}(u_2)k\ell \in K\ast L$. Therefore $f$ is a simplicial map. Since $f_V$ is also a bijection, $f$ is an automorphism.
Thus $f \in \Aut(K\ast L)$. However, note that $f$ sends vertices of $K$ to $L$ and vice versa (namely, it swaps vertices of $U_1$ and $U_2$), and therefore $f \not \in \Aut(K) \times \Aut(L)$.

For the other direction, we prove the contrapositive. Suppose that $\Aut(K\ast L) \not\cong\Aut(K) \times \Aut(L)$. Since  $\Aut(K) \times \Aut(L)$ is a subgroup of $\Aut(K\ast L)$, $|\Aut(K) \times \Aut(L)| < |\Aut(K\ast L)|$. Thus there must exist some function $f \in \Aut(K\ast L)$, with $f \not\in \Aut(K) \times \Aut(L)$. As shown earlier, $\Aut(K) \times \Aut(L)$ consists of isomorphisms on $K\ast L$ that send subcomplexes of $K$ to subcomplexes of $K$, and subcomplexes of $L$ to subcomplexes of $L$. Therefore $f$ must send some subcomplex $U_1$ of $K$ to some subcomplex $U_2$ of $L$. We first show that $U_2$ is fully connected in $L$. Notice that in $K\ast L$, $K$ (and thus any subcomplex of $K$) is joined to $L$. Therefore $U_1$ is joined to $L$, so it is joined to $L-U_2$. Since $f$ is an isomorphism, it follows that $U_2$ is joined to $L-U_2$. This means $U_2$ is fully connected in $L$. By the same reasoning, since $U_2$ is joined to $K$, and thus $K-U_1$, it follows that $U_1$ is joined to $K-U_1$, so $U_1$ is fully connected in $K$. Thus, the desired condition holds.
\end{proof}



\begin{thm} Let $K_1,K_2$ be simplicial complexes with $K_1,K_2\neq \pd^n$ or $C_n$. If $K_1, K_2$ do not contain subcomplexes $U_1, U_2$ satisfying the conditions of Proposition \ref{prop: aut product}, then $\Aut(\M(K_1\cup K_2)) \cong \Aut(K_1)\times \Aut(K_2)$.
\end{thm}

\begin{proof}
For disjoint complexes $K_1, K_2$, we have
$\M(K_1\cup K_2)  \cong \M(K_1)\ast \M(K_2)$ by Proposition \ref{prop: morse union}, hence $\Aut(\M(K_1\cup K_2))  \cong \Aut(\M(K_1)\ast \M(K_2))$. We then have $\Aut(\M(K_1\cup K_2)) \cong \Aut(\M(K_1))\times \Aut(\M(K_2))$ precisely when the condition in Proposition \ref{prop: aut product} holds. By \cite[Theorem 1]{LinSco-19}, $\Aut(\M(K_1\cup K_2)) \cong \Aut(K_1)\times \Aut(K_2)$.
\end{proof}

A similar statement, which we omit here, can be made with $K_1,K_2=\pd^n$ or $C_n$.

\section{Future directions and open questions}\label{sec: Future directions and open questions}

The many applications of the results found in Section \ref{sec: The Morse complex of cycles wedge a leaf} suggests a convenient way to determine the strong collapsibility of a simplicial complex's Morse complex via a careful study of the Hasse Diagram. The following Lemma further evidences this claim.

\begin{lem}\label{lem: Hasse poset lemma}
Let $X$ be a poset and suppose that $X=A \sqcup B$. If either $f(A)$ or $f(B)$ is strongly collapsible, then so is $f(X).$
\end{lem}

\begin{proof}
Without loss of generality, suppose that $f(A)$ is strongly collapsible.  Since $A$ is disjoint from $B$, all of the primitive vectors in $f(A)$ are compatible with all simplices of $f(X)-f(A)$. Hence, after strong collapsing $f(A)$ to some primitive vector $v$, $v$ will be compatible with all simplices of $f(X)-f(A)$, and thus dominates all vertices in $f(X)-f(A)$. Therefore $f(X)$ is also strongly collapsible.
\end{proof}

In other words, if through the addition of leaves to a simplicial complex we encounter a disjoint section of the Hasse diagram whose Morse complex is strongly collapsible, than the original simplicial complex's Morse complex is also strongly collapsible. \\

This lemma provides a potentially convenient way to determine the strong collapsibility of the Morse complex of graphs, by developing a comprehensive collection of such "disjoint sections" that are known to have strongly collapsible Morse complexes. This can be done systematically by beginning with posets of height $2$ (corresponding to graphs) whose bottom layer has $1$ node, $2$ nodes, $3$ nodes, etc. We will show the cases for $1$ and $2$ nodes in the bottom layer:\\

\textbf{Case 1:} $1$ node. The following is the only such connected poset of height $2$ whose Morse complex is strongly collapsible, as its Morse complex is a single point:

$$\begin{tikzpicture}[scale=.8,
    decoration={markings,mark=at position 0.6 with {\arrow{triangle 60}}},
    ]

\node[inner sep=1pt, circle, fill=black] (a) at (0,0) [draw] {};
\node[inner sep=1pt, circle, fill=black] (b) at (0,1) [draw] {};

\draw[-]  (a)--(b) node[midway, below] {};

\end{tikzpicture}$$

If there is more than $1$ node in the second layer of the poset, then its Morse complex will consist entirely of disjoint points, and thus would not be strongly collapsible.\\

\noindent\textbf{Case 2:} $2$ vertices. The following is the only such connected poset of height $2$ whose Morse complex is strongly collapsible:
$$
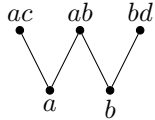
\begin{tikzpicture}[scale=.8,
    decoration={markings,mark=at position 0.6 with {\arrow{triangle 60}}},
    ]

\node[inner sep=1pt, circle, fill=black] (a) at (-0.5,0) [draw] {};
\node[inner sep=1pt, circle, fill=black] (b) at (0.5,0) [draw] {};
\node[inner sep=1pt, circle, fill=black] (c) at (0,1) [draw] {};
\node[inner sep=1pt, circle, fill=black] (d) at (1,1) [draw] {};
\node[inner sep=1pt, circle, fill=black] (e) at (-1,1) [draw] {};

\draw[-]  (a)--(e) node[midway, below] {};
\draw[-]  (a)--(c) node[midway, below] {};
\draw[-]  (b)--(c) node[midway, below] {};
\draw[-]  (b)--(d) node[midway, below] {};

\node[anchor = north ]  at (a) {{$a$}};
\node[anchor = north ]  at (b) {{$b$}};
\node[anchor = south ]  at (c) {{$ab$}};
\node[anchor = south ]  at (d) {{$bd$}};
\node[anchor = south ]  at (e) {{$ac$}};
\end{tikzpicture}$$
We now show that all other posets with 2 vertices in the bottom layer have Morse complexes that are not strongly collapsible. It is easy to verify that a connected poset of height $2$ with $2$ nodes in the bottom layer can only have $1$ node in the second layer that is connected to both nodes in the bottom layer. We now consider the casework on the number of nodes $a$ and $b$ are connected to in the second layer. We have two cases:\\

\textbf{Case 2a:} One of $a$ and $b$ is connected to more than $1$ node in the second layer.
$$
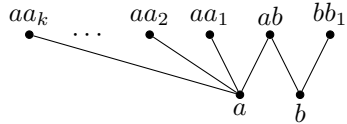
\begin{tikzpicture}[scale=.8,
    decoration={markings,mark=at position 0.6 with {\arrow{triangle 60}}},
    ]

\node[inner sep=1pt, circle, fill=black] (a) at (-0.5,0) [draw] {};
\node[inner sep=1pt, circle, fill=black] (b) at (0.5,0) [draw] {};
\node[inner sep=1pt, circle, fill=black] (c) at (0,1) [draw] {};
\node[inner sep=1pt, circle, fill=black] (d) at (1,1) [draw] {};
\node[inner sep=1pt, circle, fill=black] (e) at (-1,1) [draw] {};
\node[inner sep=1pt, circle, fill=black] (a2) at (-2,1) [draw] {};
\node[inner sep=1pt, circle, fill=black] (ak) at (-4,1) [draw] {};
\node at (-3,1) {$\cdots$};

\draw[-]  (a)--(e) node[midway, below] {};
\draw[-]  (a)--(c) node[midway, below] {};
\draw[-]  (b)--(c) node[midway, below] {};
\draw[-]  (b)--(d) node[midway, below] {};
\draw[-]  (a)--(a2) node[midway, below] {};
\draw[-]  (a)--(ak) node[midway, below] {};

\node[anchor = north ]  at (a) {{$a$}};
\node[anchor = north ]  at (b) {{$b$}};
\node[anchor = south ]  at (c) {{$ab$}};
\node[anchor = south ]  at (d) {{$bb_1$}};
\node[anchor = south ]  at (e) {{$aa_1$}};
\node[anchor = south ]  at (a2) {{$aa_2$}};
\node[anchor = south ]  at (ak) {{$aa_k$}};
\end{tikzpicture}$$
The Morse complex of this poset consists of the disjoint nodes $(a, aa_1), (a, aa_2), \dots , (a,aa_k)$ joined to the two disjoint nodes $(b,bb_1)$ and $(b,ab)$, together with $(a,ab)$ connected to $(b,bb_1)$. There is only one pair of dominating vertices, namely $(b,bb_1)$ dominates $(a,ab)$, and it can be easily shown that no other dominating pair exists. Thus the Morse complex of posets of this form are not strongly collapsible.\\

\noindent\textbf{Case 2b:} Both $a$ and $b$ are connected to more than $1$ node in the second layer.
$$
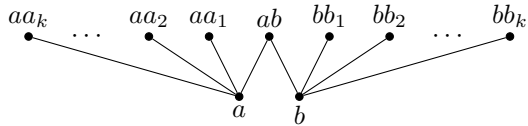
\begin{tikzpicture}[scale=.8,
    decoration={markings,mark=at position 0.6 with {\arrow{triangle 60}}},
    ]

\node[inner sep=1pt, circle, fill=black] (a) at (-0.5,0) [draw] {};
\node[inner sep=1pt, circle, fill=black] (b) at (0.5,0) [draw] {};
\node[inner sep=1pt, circle, fill=black] (c) at (0,1) [draw] {};
\node[inner sep=1pt, circle, fill=black] (d) at (1,1) [draw] {};
\node[inner sep=1pt, circle, fill=black] (e) at (-1,1) [draw] {};
\node[inner sep=1pt, circle, fill=black] (a2) at (-2,1) [draw] {};
\node[inner sep=1pt, circle, fill=black] (ak) at (-4,1) [draw] {};
\node at (-3,1) {$\cdots$};
\node[inner sep=1pt, circle, fill=black] (b2) at (2,1) [draw] {};
\node[inner sep=1pt, circle, fill=black] (bk) at (4,1) [draw] {};
\node at (3,1) {$\cdots$};

\draw[-]  (a)--(e) node[midway, below] {};
\draw[-]  (a)--(c) node[midway, below] {};
\draw[-]  (b)--(c) node[midway, below] {};
\draw[-]  (b)--(d) node[midway, below] {};
\draw[-]  (a)--(a2) node[midway, below] {};
\draw[-]  (a)--(ak) node[midway, below] {};
\draw[-]  (b)--(b2) node[midway, below] {};
\draw[-]  (b)--(bk) node[midway, below] {};

\node[anchor = north ]  at (a) {{$a$}};
\node[anchor = north ]  at (b) {{$b$}};
\node[anchor = south ]  at (c) {{$ab$}};
\node[anchor = south ]  at (d) {{$bb_1$}};
\node[anchor = south ]  at (e) {{$aa_1$}};
\node[anchor = south ]  at (a2) {{$aa_2$}};
\node[anchor = south ]  at (ak) {{$aa_k$}};
\node[anchor = south ]  at (b2) {{$bb_2$}};
\node[anchor = south ]  at (bk) {{$bb_k$}};

\end{tikzpicture}$$

The Morse complex of this poset consists of the disjoint points $(a,ab), (a,aa_1),$ $(a,aa_2), \dots , (a,aa_k)$ joined to the disjoint points $(b,ab),(b,bb_1),(b,bb_2), \dots , (b, bb_k)$, together with $(a,ab)$ joined to the disjoint points $(b,bb_1),(b,bb_2), \dots , (b, bb_k)$. Any point connected to a point of the form $(a,aa_i)$ is also connected to the other $k-1$ points of the form $(a,aa_i)$. Since all points of the form $(a,aa_i)$ are disjoint, $(a,aa_i)$ cannot dominate any other point. By symmetry, neither can any point of the form $(b,bb_i)$. By similar reasoning, it can be easily seen that $(a,ab)$ and $(b,ab)$ cannot dominate any other points. Therefore, as there are no dominating vertices in the Morse complex, it is not strongly collapsible.\\

This process can be continued to determine posets with strongly collapsible Morse complexes with greater amounts of vertices in the bottom layer. For example, the following is one such poset for $3$ vertices:
$$
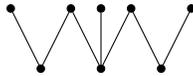
\begin{tikzpicture}[scale=.8,
    decoration={markings,mark=at position 0.6 with {\arrow{triangle 60}}},
    ]

\node[inner sep=1pt, circle, fill=black] (a) at (0,0) [draw] {};
\node[inner sep=1pt, circle, fill=black] (b) at (-1,0) [draw] {};
\node[inner sep=1pt, circle, fill=black] (c) at (1,0) [draw] {};
\node[inner sep=1pt, circle, fill=black] (aa) at (0,1) [draw] {};
\node[inner sep=1pt, circle, fill=black] (ac) at (0.5,1) [draw] {};
\node[inner sep=1pt, circle, fill=black] (ab) at (-0.5,1) [draw] {};
\node[inner sep=1pt, circle, fill=black] (cc) at (1.5,1) [draw] {};
\node[inner sep=1pt, circle, fill=black] (bb) at (-1.5,1) [draw] {};

\draw[-]  (a)--(aa) node[midway, below] {};
\draw[-]  (a)--(ac) node[midway, below] {};
\draw[-]  (c)--(ac) node[midway, below] {};
\draw[-]  (b)--(ab) node[midway, below] {};
\draw[-]  (a)--(ab) node[midway, below] {};
\draw[-]  (b)--(bb) node[midway, below] {};
\draw[-]  (c)--(cc) node[midway, below] {};

\end{tikzpicture}$$

Whether this is the only such poset remains open.\\
If sufficiently many strongly collapsible posets of any given number of vertices in the bottom layer are computed, a simple algorithm can be used to determine the strong collapsibility of the Morse complex of any simplicial complex of dimension $1$:

 \begin{algorithm}[H]
  \caption{Strong Collapsibility of Morse complex}\label{alg:Hasse1}
\verb"Input:" Simplicial Complex $K$  \\
\verb"Output:"  \textbf{True} if $\M(K)$ is strongly collapsible, $\textbf{False}$ if not

\begin{enumerate}\itemsep1pt \parskip0pt
\parsep0pt
\item[1] Compute the Hasse Diagram $\HH(K)$
\item[2] \textbf{for each} subposet $p$ of $\HH(K)$ of height $2$:
\item[3] \quad \textbf{if} $f(p)$ is known to be strongly collapsible:
\item[4] \quad\quad \textbf{return True}
\item[5] \quad \textbf{end if}
\item[6]\textbf{end for}
\item[7] \textbf{return False}
  \end{enumerate}
\end{algorithm}

\end{document}